\documentclass[a4paper]{amsart}
\usepackage{amssymb,amsmath,amsthm,verbatim}
\usepackage[parfill]{parskip}

\DeclareMathOperator{\supp}{supp}

\DeclareMathOperator{\Prob}{Prob}

\newcommand{\N}{\mathbb{N}}
\newcommand{\C}{\mathbb{C}}
\newcommand{\norm}[1]{\left\|#1\right\|}
\newcommand{\GAMMA}{\boldsymbol{\Gamma}}
\newcommand{\E}{\mathcal{E}}
\newcommand{\EQ}{\mathcal{E}_Q}
\newcommand{\V}{\mathcal{V}}
\newcommand{\oneQ}{\mathbf{1}_Q}

\title{Expanders and Property A}

\author{A. Khukhro}
\author{N.J. Wright}

\theoremstyle{plain}
\newtheorem{thm}{Theorem}[section]
\newtheorem{lemma}[thm]{Lemma}
\newtheorem{cor}[thm]{Corollary}

\theoremstyle{definition}
\newtheorem{definition}[thm]{Definition}

\begin{document}
\begin{abstract}
We give a cohomological characterisation of expander graphs, and use it to give a direct proof that expander graphs do not have Yu's property A.
\end{abstract}

\maketitle

\section{Introduction}
Property A, first introduced in \cite{Yu}, is a coarse geometric analogue of amenability. 
\begin{definition}
A discrete bounded geometry metric space $X$ has \emph{property A} if for each $x\in X$ and each $n\in \mathbb{N}$, there is an element $f_n(x) \in \Prob(X)$ with
\begin{enumerate}
\item
a sequence $S_n$ such that $\supp(f_n(x)) \subseteq B_{S_n}(x)$, and
\item
for any $R>0$, $\|f_n(x_1)-f_n(x_0)\|_{\ell^1} \rightarrow 0 $ as $n \rightarrow \infty$ uniformly on the set $\{(x_0,x_1): d(x_0,x_1)\leq R\}$.
\end{enumerate}
\end{definition}
In \cite{Yu}, Yu proves that if a metric space has property A then it is uniformly embeddable into Hilbert space. Indeed, this was the original motivation behind this definition, since a result of the same paper \cite{Yu} states that the coarse Baum--Connes conjecture holds for discrete bounded geometry metric spaces which admit a uniform embedding into Hilbert space.

There are few known examples of metric spaces which do not have property A. One such family of examples is provided by \emph{expander graphs} (cf.\ \cite{Lub}, \cite{Mar}). Informally, an expander is a sequence of highly connected graphs which have bounded valency. Expander graphs are used in computer science due to their high connectivity. They are also of theoretical interest as they provide counterexamples to the coarse Baum--Connes conjecture \cite{HLS}.

Expander graphs do not uniformly embed into Hilbert space (see for example \cite{Roe}) and so cannot have property A. In this paper we give a direct, more geometric proof that expanders do not have property A, making the connection between the two properties explicit. This is based on the observation that both the expander condition and property A can be expressed in terms of a coboundary operator which, roughly speaking, measures the size of the (co)boundary of a set of vertices. The cohomological description of property A was given in \cite{BNW}, while the cohomological description of the expander condition is introduced in this paper.

\section{Expanders and cohomology}

Let $\{\Gamma_i\}$ be a sequence of finite graphs. Abusing notation, we will also denote the vertex set by $\Gamma_i$ and the edges by $E_i$. We take the edges to be directed, with an edge connecting $x$ to $y$ if and only if there is an edge connecting $y$ to $x$. The \emph{Cheeger constant} of the graph $\Gamma_i$ is defined by $h(\Gamma_i)=\frac{1}{2}\inf \frac{|\partial F|}{|F|}$, where $F$ ranges over the non-empty subsets of $\Gamma_i$ such that $|F|\leq \frac{1}{2}|\Gamma_i|$ and $\partial F$ denotes the coboundary\footnote{This is usually referred to as the boundary of $F$, however as the map goes from vertices to edges, homologically it is a coboundary.} of $F$, i.e.\ the set of edges of $\Gamma_i$ with exactly one end point in $F$. The factor of $\frac{1}{2}$ compensates for the doubling arising from the use of directed edges. 

\begin{definition}
A finite graph $\Gamma$ is a \emph{$(k,\varepsilon)$-expander} if each vertex of $\Gamma$ has valency at most $k$, and $h(\Gamma)\geq \varepsilon$.

A sequence of finite graphs $\{\Gamma_i\}$ is called an \emph{expander sequence} if $|\Gamma_i|\rightarrow \infty$ and there exists $k,\varepsilon$ such that each $\Gamma_i$ is a $(k,\varepsilon)$-expander.
\end{definition}
It is not obvious that such sequences exist. Their existence was first proved by Pinsker, in a non-constructive way. Margulis was the first to give explicit examples of expanders, using discrete groups with property (T) \cite{Mar}.

Let $\Gamma$ be a finite graph and let $E$ denote its set of directed edges. We view $\mathbb{C}$ as the subspace of $\ell^1(\Gamma)$ consisting of constant functions, and write $\overline{f}$ for the class in $\ell^1(\Gamma)/\mathbb{C}$ represented by $f\in \ell^1(\Gamma)$.
The norm on $\ell^1(\Gamma)/\mathbb{C}$ is the quotient norm defined by $\|\overline{f}\|_{\ell^1/\mathbb{C}}=\inf_{c\in \mathbb{C}}\norm{f+c}_{\ell^1}$. 
We will write $\ell^1_0(E)$ for the subspace of $\ell^1(E)$ consisting of functions whose sum is zero. The norm on $\ell^1_0(E_i)$ is the usual $\ell_1$ norm. 
Define a coboundary map
$$d:\ell^1(\Gamma)/\mathbb{C} \longrightarrow \ell^1_0(E)$$
by $d\overline{f}(e)=f(e^+)-f(e^-)$ where $e^-$ is the starting vertex and $e^+$ is the end vertex of the directed edge $e$.

\begin{lemma}\label{bounded below}
The Cheeger constant $h(\Gamma)$ is at least $\frac{\varepsilon}{2}$ if and only if $\|d\overline{f}\|_{\ell^1}\geq \varepsilon \|\overline{f}\|_{\ell^1/\mathbb{C}}$ for every $\overline{f}\in \ell^1(\Gamma)/\mathbb{C}$.
\end{lemma}
\begin{proof}
Suppose $\|d\overline{f}\|_{\ell^1}\geq \varepsilon \|\overline{f}\|_{\ell^1/\mathbb{C}}$ for every $\overline{f}\in \ell^1(\Gamma)/\mathbb{C}$. Then in particular, for any subset $F\subset \Gamma$ such that $|F|\leq \frac{1}{2}|\Gamma|$ we have $\|d\overline{\chi_F}\|_{\ell^1}\geq \varepsilon \|\overline{\chi_F}\|_{\ell^1/\mathbb{C}}$, where $\chi_F$ denotes the characteristic function of $F$. It is clear that $\|d\overline{\chi_F}\|_1$ is equal to $|\partial F|$, the coboundary of the set $F$ (recall that we are taking our edges to be directed). Also, since $|F|\leq \frac{1}{2}|\Gamma|$, we have $$\sum_{\gamma \in \Gamma} |\chi_F(\gamma)+c|= \sum_{\gamma\in F}|1+c| +\sum_{\gamma \notin F}|c| \geq \sum_{\gamma\in F}1 - \sum_{\gamma\in F}|c|+\sum_{\gamma \notin F}|c|\geq  \sum_{\gamma\in F}1.$$ From this, we can see that the infimum over $c\in \mathbb{C}$ of $\sum_{\gamma \in \Gamma} |\chi_F(\gamma)+c|$ is achieved when $c=0$ and so we have $\|\overline{\chi_F}\|_{\ell^1/\mathbb{C}}=|F|$. Hence for every $F$ with $|F|\leq \frac{1}{2}|\Gamma|$, we have $|\partial F|\geq \varepsilon |F|$ and so $h(\Gamma)\geq \frac{\varepsilon}{2}$.

Suppose now that $h(\Gamma)$ is at least $\frac{\varepsilon}{2}$. 
Given $\overline{f}\in \ell^1(\Gamma)/\mathbb{C}$, pick an $f'\in \ell^1(\Gamma)$ which takes positive values on each element of $\Gamma$ and such that $\overline{f'}=\overline{f}$.
We can write $f'$ as $\sum a_j\chi_{F_j}$ for some nested collection of subsets $F_1 \subset F_2 \subset ...\subset F_n$ of $\Gamma$ and coefficients $a_j>0$. Now $\|d\overline{f}\|_{\ell^1}=\|d\overline{f'}\|_{\ell^1}$ is equal to $\sum a_j \|d\overline{\chi_{F_j}}\|_{\ell^1}$ since the $F_j$ are nested. Hence
$$
\|d \overline{f}\|_{\ell^1}\geq \sum_j a_j \|d \overline{\chi_{F_j}}\|_{\ell^1}=\sum_{j}a_j |\partial F_j|.$$
Let $F_j^c$ denote the complement of $F_j$ in $\Gamma$. Since $h(\Gamma)\geq \frac{\varepsilon}{2}$, when $|F_j|\leq \frac{1}{2}|\Gamma|$ we have $|\partial F_j|\geq \varepsilon |F_j|=\|\overline{\chi_{F_j}}\|_{\ell^1/\mathbb{C}}$, while for $|F_j|> \frac{1}{2}|\Gamma|$ we have
$$|\partial F_j|=|\partial F_j^c|\geq \varepsilon |F_j^c|=\varepsilon \|\overline{\chi_{F_j^c}}\|_{\ell^1/\mathbb{C}}=\varepsilon \|\overline{1-\chi_{F_j}}\|_{\ell^1/\mathbb{C}}=\varepsilon \|\overline{\chi_{F_j}}\|_{\ell^1/\mathbb{C}},$$
and so
$$
\|d \overline{f}\|_{\ell^1}\geq  \varepsilon\sum_{j}a_j  \|\overline{\chi_{F_j}}\|_{\ell^1/\mathbb{C}}
\geq \varepsilon \|\overline{\sum_j a_j \chi_{F_j}}\|_{\ell^1/\mathbb{C}}=
\varepsilon \|\overline{f}\|_{\ell^1/\mathbb{C}}.
$$
This completes the proof.
\end{proof}

The map $\ell^1(\Gamma)\to\ell^1_0(\Gamma)$ taking a function $f\in \ell^1(\Gamma)$ to $g=f-\frac{1}{|\Gamma|}\sum_{\beta\in\Gamma}f(\beta)$ has kernel $\C$, and hence induces an isomorphism from $\ell^1(\Gamma)/\mathbb{C}$ to $\ell^1_0(\Gamma)$. This map has norm at most $2$ since
$$\norm{g}_{\ell^1}=\sum_{\gamma\in\Gamma}|f(\gamma)-\frac{1}{|\Gamma|}\sum_{\beta\in\Gamma}f(\beta)|\leq \sum_{\gamma\in\Gamma}|f(\gamma)|+\sum_{\beta\in\Gamma}|f(\beta)|=2\norm{f}_{\ell^1}$$
while the inverse is given by the inclusion of $\ell^1_0(\Gamma)$ in $\ell^1(\Gamma)$ which has norm 1. Hence identifying $\ell^1(\Gamma)/\mathbb{C}$ with $\ell^1_0(\Gamma)$, the norms differ by a factor of at most 2.

We now move on to the definition of the cohomology which detects expander sequences. Let $\{\Gamma_i\}_{i\in \mathbb{N}}$ be a sequence of graphs. We denote by $\prod^\infty_{i\in\N}\ell^1(\Gamma_i)$ the space of bounded elements of the direct product. That is, $\prod^\infty_{i\in\N}\ell^1(\Gamma_i)$ is the space of functions from $\coprod_i\Gamma_i$ to $\C$, such that the $\sup$-$\ell^1$-norm
$$\norm{f}=\sup_{i\in\N} \norm{f|_{\Gamma_i}}_{\ell^1}$$
is finite. We define a summation map $\sigma_0:\prod^\infty_{i\in\N}\ell^1(\Gamma_i)\to \ell^\infty(\N)$ by $\sigma_0(f)(i)=\sum_{x\in \Gamma_i} f(x)$. Similarly $\prod^\infty_{i\in\N}\ell^1(E_i)$  is the space of functions on $\coprod_i E_i$ with finite $\sup$-$\ell^1$-norm, and we define $\sigma_1:\prod^\infty_{i\in\N}\ell^1(E_i)\to \ell^\infty(\N)$ by $\sigma_1(f)(i)=\sum_{x\in E_i} f(x)$.

We define
$$C^0(\{\Gamma_i\})=\ker(\sigma_0),\quad C^1(\{\Gamma_i\})=\ker(\sigma_1).$$
Note that $C^0(\{\Gamma_i\})$ consists of functions whose restriction to each $\Gamma_i$ lies in $\ell^1_0(\Gamma_i)$, and $C^1(\{\Gamma_i\})$ consists of functions whose restriction to each $E_i$ is in $\ell^1_0(E_i)$. Hence combining the coboundary maps on each component yields a coboundary map $d:C^0(\{\Gamma_i\})\to C^1(\{\Gamma_i\})$, and it is easy to see that this is bounded. In the spirit of \cite{BNW}, our cohomological description of the expander condition is given by completing this cochain complex.

\begin{definition}[{\cite[Def.\ 3.1]{BNW}}]
The \emph{quotient completion} of a pre-Fr\'echet space $V$ (a space equipped with a countable family of seminorms $\norm{\cdot}_j$) is the space $V_Q=\ell^\infty(\N,V)/c_0(\N,V)$ of bounded sequences in $V$ modulo sequences vanishing at infinity.
\end{definition}

For simplicity we suppose that the seminorms are monotonic, that is $\norm{\cdot}_i\leq \norm{\cdot}_j$ for $i<j$. We note the following useful property of this completion.

\begin{lemma}\label{quotient completion}
Let $T:V\to W$ be a bounded map from a normed spaced $V$ to a pre-Fr\'echet space $W$. Then $T$ is bounded below if and only if the induced map $T^Q:V_Q\to W_Q$ \cite[Prop.\ 3.3]{BNW} is injective.
\end{lemma}

\begin{proof}
One direction is obvious: if $T$ is bounded below then $T^Q$ is also bounded below hence injective. For the converse suppose that $T$ is not bounded below. This means that for each seminorm $\norm{\cdot}_{j,W}$ for $W$ and all $\varepsilon>0$ there exists $v$ in $V$ with $\norm{Tv}_{j,W}<\varepsilon\norm{v}_{V}$. Hence we can find a sequence $v_n\in V$ with $\norm{v_n}_{V}=1$ and $\norm{Tv_n}_{n,W}<\frac 1n$. As the sequence $v_n$ is bounded, it determines an element $v$ of $V_Q$. Its image under $T^Q$ is given by the sequence $Tv_n$, and since for $n\geq j$ we have $\norm{Tv_n}_{j,W}\leq \norm{Tv_n}_{n,W}<\frac 1n$, we have $Tv_n\in c_0(\N,W)$. Hence $T^Qv=0$, so $T^Q$ is not injective.
\end{proof}

We remark that the lemma is not true in general if $V$ is a pre-Fr\'echet space. Whilst for $T$ not bounded below there still exists a sequence $v_n$ not tending to zero such that $Tv_n\to 0$, there may be no \emph{bounded} sequence with these properties.

\medskip

We now give our cohomological description of the expander condition. Let $C^p_Q(\{\Gamma_i\})$ denote the quotient completion of $C^p(\{\Gamma_i\})$ for $p=0,1$. The extension of the coboundary map $d$ to the completion we again denote by $d$.

\begin{definition}
The \emph{Cheeger cohomology} of a sequence of graphs $\{\Gamma_i\}$, denoted $H_{h}^*(\{\Gamma_i\})$ is the cohomology of the cochain complex $(C^p_Q(\{\Gamma_i\}),d)$.
\end{definition}

We remark that $C^p_Q(\{\Gamma_i\})$ is the kernel of the induced map $\sigma_p^Q$, since the quotient completion preserves exactness (cf.\ \cite{BNW}).

\begin{thm}
Let $\{\Gamma_i\}_{i\in \mathbb{N}}$ be a sequence of finite graphs with bounded valency. Then $\{\Gamma_i\}$ is an expander sequence if and only if $H_{h}^0(\{\Gamma_i\})$ vanishes.
\end{thm}

\begin{proof}
Using Lemma \ref{bounded below} and the identification of $\ell^1(\Gamma_i)/\C$ with $\ell^1_0(\Gamma_i)$, the graphs $\{\Gamma_i\}$ form an expander sequence if and only if there exists $\varepsilon>0$ such that for each graph $\Gamma_i$ the coboundary map $d:\ell^1_0(\Gamma_i)\to \ell^1_0(E_i)$ is $\varepsilon$-bounded below. The individual coboundary maps are bounded below by a common $\varepsilon$ if and only if the map $d:C^0(\{\Gamma_i\})\to C^1(\{\Gamma_i\})$ is bounded below. By Lemma \ref{quotient completion} this is equivalent to injectivity of the coboundary map $d:C^0_Q(\{\Gamma_i\})\to C^1_Q(\{\Gamma_i\})$ on the completed complex. Hence the graphs $\{\Gamma_i\}$ form an expander sequence if and only if $H_{h}^0(\{\Gamma_i\})=0$.
\end{proof}

\section{Symmetrisation of property A}

In this section we recall one of the cohomological characterisations of property A from \cite{BNW}, and prove a symmetrisation result. Throughout this section, let $X$ denote a metric space. At certain points we will require $X$ to be a discrete, bounded geometry space, that is, for each $R>0$ there exists $N$ such that for all $x\in X$ the ball of radius $R$ about $x$ contains at most $N$ points.

\begin{definition}
An $X$-module is a triple $\V=(V,\norm{\cdot},\supp)$, where $V$ is a Banach space with norm $\norm{\cdot}$ and $\supp$ is a function from $V$ to the power set of $X$ such that
\begin{enumerate}
\item $\supp(v)=\emptyset$ if  $v=0$,
\item $\supp(v+w)\subseteq  \supp(v)\cup \supp(w)$ for every $v,w\in V$,
\item $\supp(\lambda v)=\supp(v)$ for every $v\in V$ and every $\lambda\not=0$, 
\item if $v_n$ is a sequence converging to $v$ then $\supp(v)\subseteq \overline{\bigcup\limits_n \supp(v_n)}$.
\end{enumerate}
\end{definition}

Let $\E^p(X,\V)$ denote the space of functions $\phi$ from $X^{p+1}$ to $V$ such that for all $R>0$ the function $\phi$ is bounded on
$$\Delta_R^{p+1}=\{(x_0,\dots x_p)\in X^{p+1}: d(x_i,x_j)\leq R\text{ for all }i,j \}$$
and there exists $S>0$ such that if $\mathbf{x}=(x_0,\dots x_p)\in \Delta_R^{p+1}$ then $\supp (\phi(\mathbf{x}))\subseteq B_S(x_i)$ for all $i$.

The space $\E^p(X,V)$ is equipped with the family of seminorms
$$\norm{\phi}_R=\sup\{\norm{\phi(x)}_V : x\in \Delta^{p+1}_R\}.$$
In \cite{BNW} this is denoted by $\E^{p,-1}(X,\V)$, being part of a bicomplex, however for simplicity we drop the $-1$ from our notation. We note that $\E^0(X,\V)$ is a normed space, since in dimension zero the norms are independent of $R$.

Let $\E^p_Q(X,\V)$ denote the quotient completion of $\E^p(X,\V)$. The usual formula $D\phi(x_0,\dots,x_{p+1})=\sum\limits_{i=0}^{p+1}(-1)^i \phi(x_0,\dots\hat{x}_i,\dots,x_{p+1})$ yields a coboundary map from $\E^p(X,\V)$ to $\E^{p+1}(X,\V)$, and the extension of $D$ to the completion we again denote by $D$.

The \emph{controlled cohomology} $H_Q^*(X,\V)$ is the cohomology of the completed complex $(\E^p_Q(X,\V),D)$.

By \cite[Theorem 7.2]{BNW} the space $X$ has property A if and only if the class $[\oneQ]\in H_Q^0(X,\C)$ is in the image of the map $\pi_*:H_Q^0(X,\ell^1(X))\to H_Q^0(X,\C)$ induced by the summation map $\pi:\ell^1(X)\to \C$. Here the module $\ell^1(X)$ is equipped with the usual support function, while all elements of $\C$ are defined to have empty support.

We now compare $\ell^1$ and $\ell^2$ coefficients. We define maps $\alpha:\ell^1(X)\to\ell^2(X)$ and $\beta:\ell^2(X)\to\ell^1(X)$ by
$$\alpha(\eta)(x)=\sqrt{|\eta(x)|}\text{ for $\eta\in\ell^1(X)$},\quad\beta(\xi)(x)=|\xi(x)|^2\text{ for $\xi\in\ell^2(X)$}.$$
Note that $\norm{\alpha(\eta)}^2_{\ell^2}=\norm{\eta}_{\ell^1}$ and $\norm{\beta(\xi)}_{\ell^1}=\norm{\xi}_{\ell^2}^2$.

\begin{lemma}\label{alphabeta}
Let $\alpha,\beta$ be defined as above. Then the compositions with $\alpha$ and $\beta$, yield maps $\E^p(X,\ell^1(X))\to \E^p(X,\ell^2(X))$ and $\E^p(X,\ell^2(X))\to \E^p(X,\ell^1(X))$ which extend in the natural way to maps $\alpha_*,\beta_*$ on the quotient completions. Moreover these maps take $0$-cocycles to $0$-cocycles.
\end{lemma}

\begin{proof}
The identity $\norm{\alpha(\eta)}^2_{\ell^2}=\norm{\eta}_{\ell^1}$ shows that for $\phi_n$ a bounded sequence in $\E^p(X,\ell^1(X))$, the sequence $\alpha\circ\phi_n\in\E^p(X,\ell^2(X))$ is also bounded. Hence, as composition with $\alpha$ preserves supports, $\alpha\circ\phi_n$ defines an element in the quotient completion. We note that the inequalities
$$|\sqrt{|\eta(z)|}-\sqrt{|\eta'(z)|}|\leq \sqrt{|\eta(z)|-|\eta'(z)|}\leq \sqrt{|\eta(z)-\eta'(z)|}$$
imply that $\norm{\alpha(\eta)-\alpha(\eta')}^2_{\ell^2}\leq \norm{\eta-\eta'}_{\ell^1}$. It follows that if $\phi_n'$ is another bounded sequence in $\E^p(X,\ell^1(X))$ such that $\norm{\phi_n-\phi_n'}_R\to 0$, then $\norm{\alpha\circ\phi_n-\alpha\circ\phi'_n}_R\to 0$, and so the element of $\EQ^p(X,\ell^2(X))$ obtained by composition with $\alpha$ is independent of the choice of representative of element of $\EQ^p(X,\ell^1(X))$. Thus we have a well-defined map $\alpha_*:\EQ^p(X,\ell^1(X))\to \EQ^p(X,\ell^2(X))$.

The estimate $\norm{\alpha(\eta)-\alpha(\eta')}^2_{\ell^2}\leq \norm{\eta-\eta'}_{\ell^1}$ also yields
\begin{align*}
\norm{D\alpha(\phi_n)(x_0,x_1)}^2_{\ell^2}&=\norm{\alpha(\phi_n(x_1))-\alpha(\phi_n(x_0))}^2_{\ell^2}\\
&\leq \norm{\phi_n(x_1)-\phi_n(x_0)}_{\ell^1}\\
&=\norm{D\phi_n(x_0,x_1)}_{\ell^1}
\end{align*}
for $\phi_n$ a bounded sequence in $\EQ^0(X,\ell^1(X))$. Hence $\alpha_*$ takes 0-cocycles to 0-cocycles.

The argument for $\beta_*$ is similar, using the identity $\norm{\beta(\xi)}_{\ell^1}=\norm{\xi}_{\ell^2}^2$ and the estimate $\norm{\beta(\xi)-\beta(\xi')}_{\ell^1}\leq \norm{\xi-\xi'}_{\ell^2}(\norm{\xi}_{\ell^2}+\norm{\xi'}_{\ell^2})$ which follows from
$$
\bigl||\xi(x)|^2-|\xi'(x)|^2\bigr|=\bigl||\xi(x)|-|\xi'(x)|\bigr|\bigl(|\xi(x)|+|\xi'(x)|\bigr)\leq\bigl|\xi(x)-\xi'(x)\bigr|\bigl(|\xi(x)|+|\xi'(x)|\bigr)
$$
by the Cauchy-Schwartz inequality.
\end{proof}

We now prove a symmetrisation result. Note that we will omit norm subscripts where this does not cause confusion.

For an element $\phi$ of $\E^0_Q(X,\ell^1(X))$ or $\E^0_Q(X,\ell^2(X))$ we say $\phi$ is \emph{symmetric} if it can be represented by a sequence $\phi_n$ such that $\phi(x)(z)$ is real and $\phi_n(x)(z)=\phi_n(z)(x)$ for all $x,z\in X$. We say that $\phi$ is \emph{everywhere unital} if $\lim_{n\to\infty} \norm{\phi_n(x)}=1$ for all $x\in X$ (note that this limit is independent of the choice of representative sequence).

\begin{thm}\label{symmetrisation}
Let $X$ be a bounded geometry metric space. The following are equivalent:
\begin{enumerate}
\item $X$ has property A;
\item There is a cocycle $\phi\in\E^0_Q(X,\ell^1(X))$ such that $\pi_*(\phi)=\oneQ$;
\item There is a symmetric cocycle $\phi\in\E^0_Q(X,\ell^1(X))$ such that $\pi_*(\phi)=\oneQ$;
\item There is a symmetric cocycle $\psi\in\E^0_Q(X,\ell^2(X))$ such that $\psi$ everywhere unital.
\end{enumerate}
\end{thm}
\begin{proof}
The equivalence of (1) and (2) is \cite[Theorem 7.2]{BNW}.

First we prove (2) $\implies$ (4). Suppose there exists a cocycle $\phi\in\E^0_Q(X,\ell^1(X))$ such that $\pi_*(\phi)=\oneQ$. We consider $\alpha_*\phi$. Choosing a representative sequence $\phi_n$ for $\phi$ we note that $\norm{\alpha(\phi_n(x))}^2=\norm{\phi_n(x)}\geq 1$ for all $x$ since $\pi(\phi(x))=1$. Let $\theta_n(x)=\frac1{\norm{\alpha(\phi_n(x))}}\alpha(\phi_n(x))$. We know that $\alpha_*\phi$ is a cocycle. The estimate
\begin{align*}
\norm{\frac1{\norm{\xi}}\xi-\frac1{\norm{\xi'}}\xi'}&\leq\frac{\norm{\xi-\xi'}}{\norm{\xi}}+\left|\frac1{\norm{\xi}}-\frac1{\norm{\xi'}}\right|\norm{\xi'}=\frac{\norm{\xi-\xi'}+|\norm{\xi'}-\norm{\xi}|}{\norm{\xi}}\\
&\leq 2\norm{\xi-\xi'}
\end{align*}
for $\xi\in \ell^2(X)$ with $\norm{\xi}\geq 1$, shows that $D\theta_n\to 0$, i.e.\ $\theta$ again determines a cocycle.

Consider the operators $T_n:\ell^2(X)\to \ell^2(X)$ defined by $(T_n\xi)(y)=\sum\limits_{x\in X}\theta_n(x)(y)\xi(x)$. The support condition on $\theta_n$ provides an $S_n>0$ such that $\theta_n(x)$ is supported in $B_{S_n}(x)$, and bounded geometry gives a bound $N_n$ on the size of these balls, hence the operators $T_n$ are bounded. The support condition also shows that these operators have finite propagation, and thus they are elements of the uniform Roe algebra of $X$. Consider $T'_n=(T^*_nT_n)^{1/2}$. This lies in the uniform Roe algebra since $T_n$ does, and hence for each $n$ we can find another self-adjoint operator $T_n''$ with $T_n''$ of finite propagation and $\norm{T_n''-T_n'}\to 0$ as $n\to\infty$.

Define $\psi_n(x)= T''_n (\delta_x)$. We note that for $\xi\in \ell^2(X)$ we have
$$\langle T_n\xi,T_n\xi\rangle=\langle T_n^*T_n\xi,\xi\rangle=\langle (T_n')^2\xi,\xi\rangle=\langle T_n'\xi,T_n'\xi\rangle$$
so $\norm{T_n\xi}=\norm{T_n'\xi}$ for all $\xi$. We have $\norm{T'_n (\delta_x)}=\norm{T_n (\delta_x)}=\norm{\theta_n(x)}=1$. Hence $\norm{\psi_n(x)}=\norm{T''_n (\delta_x)}\to 1$ as $n\to\infty$. Finite propagation of $T''_n$ provides the support condition for $\psi_n$ and so $\psi_n$ gives an everywhere unital element of $\EQ^0(X,\ell^2(X))$. To see that $\psi$ is a cocycle note that $\norm{D\theta_n(x_0,x_1)}=\norm{T_n(\delta_{x_1}-\delta_{x_0})}=\norm{T'_n(\delta_{x_1}-\delta_{x_0})}$ and $\norm{D\psi_n(x_0,x_1)}=\norm{T''_n(\delta_{x_1}-\delta_{x_0})}$. As $T''_n-T'_n\to 0$, $D\theta_n\to 0$ implies $D\psi_n\to 0$.

As $T''_n$ is self-adjoint, we have $\psi_n(x)(z)=\langle T''_n \delta_x, \delta_z\rangle=\langle \delta_x, T''_n\delta_z\rangle=\overline{\psi_n(z)(x)}$. To make $\psi_n$ symmetric it therefore suffices to ensure that $\psi_n(x)(z)$ is real. For an operator $T:\ell^2(X)\to \ell^2(X)$, let $\overline{T}$ denote the operator defined by $\overline{T}\xi=\overline{T\overline{\xi}}$ where $\overline{\xi}$ denotes the entry-wise complex conjugate of $\xi$. As $\theta_n$ is real, it follows that $\overline{T_n}=T_n$, and hence $\overline{T_n^*T_n}=\overline{T_n}^*\overline{T_n}=T_n^*T_n$, hence as $T_n^*T_n=T_n'^2$ we have $\overline{T'_n}^2=\overline{T_n'^2}=T_n^*T_n$. Since the positive square-root $T_n'$ of $T_n^*T_n$ is unique we have $\overline{T_n'}=T_n'$. Without loss of generality we may assume that $\overline{T_n''}=T_n''$, since replacing $T_n''$ with its real part $\frac 12(T_n''+\overline{T_n''})$ reduces the distance from $T_n'$. Hence we have $\psi_n(x)(z)=\langle T''_n \delta_x, \delta_z\rangle$ real, so we have proved (4).

(4) $\implies$ (3) is immediate from Lemma \ref{alphabeta}: given $\psi$, we take $\phi=\beta_*\psi$. Symmetry is preserved and as $\psi$ is everywhere unital, the same holds for $\phi$. So, as $\phi$ is non-negative, we have $\pi_*\phi=\oneQ$.

(3) $\implies$ (2) is trivial.
\end{proof}

\section{Expanders do not have property A}

Let $\GAMMA$ be a disjoint union of graphs $\{\Gamma_i\}_{i\in \mathbb{N}}$ equipped with a proper metric such that the restriction to each component $\Gamma_i$ is the graph metric on $\Gamma_i$, and such that the distance between $\Gamma_i$ and its complement $\Gamma_i^c$ tends to infinity as $i\rightarrow \infty$. If $\GAMMA$ has property A then there is a cocycle $\phi\in\E^0_Q(\GAMMA,\ell^1(\GAMMA))$ with $\pi_*(\phi)=\oneQ$, while if $\{\Gamma_i\}$ is an expander sequence then $H_{h}^0(\{\Gamma_i\})$ is zero. We will show that these two cohomological conditions are contradictory. This implies that expanders cannot have property A.

\begin{thm}\label{homology A implies not exp}
Let $\GAMMA$ be a disjoint union of graphs $\Gamma_i$ with bounded valency, such that $d(\Gamma_i,\Gamma_i^c)\to \infty$ and $|\Gamma_i|\to \infty$ as $i\to\infty$. If there exists a cocycle $\phi\in\E^0_Q(\GAMMA,\ell^1(\GAMMA))$ such that $\pi_*(\phi)=\oneQ$ then $H_{h}^0(\{\Gamma_i\})$ is non-zero.
\end{thm}

\begin{proof}
Suppose there exists a cocycle $\phi\in\E^0_Q(\GAMMA,\ell^1(\GAMMA))$ such that $\pi_*(\phi)=\oneQ$. We will use this to construct a non-zero cocycle in $C^0_Q(\{\Gamma_i\})$ thus proving that $H_{h}^0(\{\Gamma_i\})$ is non-zero. By Theorem \ref{symmetrisation} we may assume that $\phi$ is a symmetric cocycle. 

For each $n\in \N$ the controlled support condition provides an $S_n>0$ such that for each $x\in \GAMMA$, the support of $\phi_n(x)$ lies in $B_{S_n}(y)$. As the distance between components tends to $\infty$, if $i$ is sufficiently large then the distance between $\Gamma_i$ and the other components of $\GAMMA$ exceeds $S_n$. Hence there exists $j_n$ such that if $i\geq j_n$ then $\phi_n(x)$ is supported in $\Gamma_i$ for all $x\in \Gamma_i$.

For each $i,n$, we choose a vertex $e^i_n \in \Gamma_i$ so that the infimum of $\!\!\!\!\!\sum\limits_{(x_0,x_1)\in E_i}\!\!\!\!\! |D\phi_n(x_0,x_1)(z)|$ over all $z\in \Gamma_i$ is realised at $z=e^i_n$, where $E_i$ denotes the set of edges of $\Gamma_i$. Note that the infimum is actually a minimum, since each $\Gamma_i$ is finite, and so such an $e^i_n$ exists. For $i\geq j_n$ we define $f^i_n\in \ell^1\Gamma_i$ by $f^i_n(x)=\phi_n(x)(e^i_n)-\frac{1}{|\Gamma_i|}$, and for $i<j_n$ we define $f^i_n$ to be 0. By symmetry of $\phi_n$, when $i\geq j_n$ we have
$$\sum_{x\in \Gamma_i}|f^i_n(x)|=\sum_{x\in \Gamma_i}|\phi_n(e^i_n)(x)-\frac{1}{|\Gamma_i|}|\leq \norm{\phi_n(e^i_n)}_{\ell^1}+1.$$
This is bounded in $i,n$, hence $f_n=(f^1_n, f^2_n,\dots)$ defines an element $f$ in the quotient completion of $\prod^\infty_{i\in\N}\ell^1(\Gamma_i)$. We will show that this is a non-zero cocycle in $C^0_Q(\{\Gamma_i\})$.

For $i<j_n$ we have $\sigma_0(f_n)(i)=\sum_{x\in \Gamma_i} f^i_n(x)=0$, while for $i\geq j_n$ we have
\begin{equation*}
\sum_{x\in \Gamma_i} f^i_n(x)=\sum_{x\in \Gamma_i} \Bigl(\phi_n(x)(e^i_n)-\frac1{|\Gamma_i|}\Bigr)=\sum_{x\in \Gamma_i} \Bigl(\phi_n(e^i_n)(x)-\frac1{|\Gamma_i|}\Bigr)= \pi_*(\phi_n)(e^i_n)-1.
\end{equation*}
by symmetry of $\phi_n$. Since $\pi_*(\phi)=\oneQ$, the sequence $\pi_*(\phi_n)(e^i_n)-1$  tends to zero (uniformly in $i$) as $n\to\infty$. Thus $\sigma_0^Q(f)=0$, so $f$ is an element of $C^0_Q(\{\Gamma_i\})$.

Recalling that the valencies of the $\Gamma_i$ are uniformly bounded, we have a bound $N_n$ on the cardinality of the balls $B_{S_n}(e^i_n)$. As $\phi_n(e^1_n)(x)=0$ outside $B_{S_n}(e^i_n)$, when $i\geq j_n$ we have the following lower bound for the $\ell^1$-norm of $f^i_n$:
\begin{equation*}
\|f^i_n\|_{\ell^1}\geq \sum_{x\in \Gamma_i \backslash B_{S_n}(e^i_n)}\frac{1}{|\Gamma_i|}\geq \frac{|\Gamma_i|-N_n}{|\Gamma_i|}=1-\frac{N_n}{|\Gamma_i|}.
\end{equation*}
Hence $\norm{f_n}_{\ell^1}\geq 1$ for all $n$. In particular $\norm{f_n}_{\ell^1}$ does not tend to zero, so $f$ is a \emph{non-zero} element of $C^0_Q(\{\Gamma_i\})$.

It remains to verify that $f$ is a cocycle. We apply the coboundary operator $d$ to $f^i_n$. This clearly vanishes when $i<j_n$, while for $i\geq j_n$ we have
$$d f^i_n(x_0,x_1)=f^i_n(x_1)-f^i_n(x_0)=D\phi(x_0,x_1)(e^i_n).$$
Our choice of $e^i_n$ now comes into play. Let $k$ be an upper bound on the valency of the graphs, so that $|E_i|/|\Gamma_i|\leq k$ for all $i$. Then we have
\begin{align*}
\|d f^i_n\|_{\ell^1}&\leq  \sum_{(x_0,x_1)\in E_i} |D\phi(x_0,x_1)(e^i_n)|\\
&=\frac{1}{|\Gamma_i|}\sum_{z\in \Gamma_i} \sum_{(x_0,x_1)\in E_i} |D\phi(x_0,x_1)(e^i_n)|\\
&\leq \frac{1}{|\Gamma_i|}\sum_{z\in \Gamma_i} \sum_{(x_0,x_1)\in E_i} |D\phi(x_0,x_1)(z)|\\
&\leq k\|D\phi_n\|_{R=1}
\end{align*}
as $\sum_{z\in \Gamma_i} |D\phi(x_0,x_1)(z)|\leq \|D\phi_n\|_{R=1}$. This tends to zero as $n\to\infty$ since $\phi$ is a cocycle. Hence $df=0$, so $f$ is a non-zero cocycle and $H_{h}^0(\{\Gamma_i\})$ is non-zero.
\end{proof}

Since property A is equivalent to existence of a cocycle $\phi\in\E^0_Q(X,\ell^1(X))$ such that $\pi_*(\phi)=\oneQ$, and a sequence of graphs is an expander if and only if $H_{h}^0(\{\Gamma_i\})$ vanishes we obtain the following immediate corollary to Theorem \ref{homology A implies not exp}.

\begin{cor}
Let $\GAMMA$ be the disjoint union of an expander sequence, with metric as above. Then $\GAMMA$ does not have property A.
\end{cor}

\end{document}